\documentclass[11pt,english]{article}

\usepackage{amsfonts,babel,amssymb,amsmath}
\usepackage{graphicx}
\usepackage{verbatim}
\usepackage{color}

\newcommand{\G}{\Gamma}
\newcommand{\lam}{\lambda}
\newcommand{\tG}{{\tilde\Gamma}}
\newcommand{\tGk}{{\tilde\Gamma_k}}
\newcommand{\g}{\gamma}
\newcommand{\tg}{{\tilde\gamma}}
\newcommand{\tgk}{{\tilde\gamma_k}}
\newcommand{\tphi}{\tilde\phi}
\newcommand{\Gc}{{\Gamma^c}}
\newcommand{\Gck}{{\Gamma^c_k}}
\newcommand{\tHp}{\tilde H^{1/2}}    
\newcommand{\tHm}{\tilde H^{-1/2}}

\newcommand{\R}{\mbox{\rm I\kern-.18em R}}
\newcommand{\tR}{\mbox{\tiny\rm I\kern-.18em R}}
\newcommand{\fR}{\mbox{\footnotesize\rm I\kern-.18em R}}
\newcommand{\sR}{\mbox{\small\rm I\kern-.18em R}}
\newcommand{\N}{\mbox{\rm I\kern-.18em N}}
\newcommand{\dist}{\mathop{\rm dist}\nolimits}

\newcommand{\spn}{\mathop{\rm span}\nolimits}

\newcommand{\<}{\langle}
\renewcommand{\>}{\rangle}

\newcommand{\supp}{\mathop{\rm supp}\nolimits}
\newcommand{\meas}{\mathop{\rm meas}\nolimits}

\newcommand{\CB}{{\cal B}}
\newcommand{\CE}{{\cal E}}
\newcommand{\CF}{{\cal F}}
\newcommand{\CC}{{\cal C}}
\newcommand{\CN}{{\cal N}}

\newcommand{\CT}{{\cal T}}

\newcommand{\bn}{{\bf n}}

\newtheorem{theorem}{Theorem}[section]

\newtheorem{remark}[theorem]{Remark}
\newtheorem{lemma}[theorem]{Lemma}

\newtheorem{prop}[theorem]{Proposition}
\newenvironment{proof}{\noindent{\em Proof}. }{\hfill$\Box$\vspace{3pt}}

\newcommand{\be}{\begin{equation}}
\newcommand{\ee}{\end{equation}}
\newcommand{\bea}{\begin{eqnarray}}
\newcommand{\eea}{\end{eqnarray}}
\newcommand{\beas}{\begin{eqnarray*}}
\newcommand{\eeas}{\end{eqnarray*}}
\newcommand{\ba}{\begin{array}}
\newcommand{\ea}{\end{array}}

\newcommand{\rref}[1]{{\rm \ref{#1}}}
\newcommand{\rcite}[1]{{\rm \cite{#1}}}

\newcommand{\mysection}[2]
{
\section{#1} \label{#2}
\setcounter{equation}{0}
\setcounter{figure}{0}
\setcounter{table}{0}
}

\title{Radial basis functions for the solution of hypersingular operators
       on open surfaces\thanks{Supported by FONDECYT-Chile
under grant number 1110324 and UNSW FRG Grant number
PS24436.}}

\author{Norbert Heuer
\thanks{Facultad de Matem\'aticas,
        Pontificia Universidad Cat\'olica de Chile,
        Avenida Vicu\~na Mackenna 4860, Santiago, Chile.
        email: {\tt nheuer@mat.puc.cl}.}
\and Thanh Tran
\thanks{School of Mathematics and Statistics,
        The University of New South Wales,
        Sydney 2052, Australia.
        email: {\tt thanh.tran@unsw.edu.au}.}
}

\begin{document}
\date{}
\maketitle

\bigskip
\begin{abstract}
We analyze the approximation by radial basis functions
of a hypersingular integral equation on an open surface.
In order to accommodate the homogeneous essential
boundary condition along the surface boundary, scaled radial
basis functions on an extended surface and Lagrangian multipliers
on the extension are used.
We prove that our method converges quasi-optimally.
Approximation results for scaled radial basis functions indicate
that, for highly regular radial basis functions, the achieved
convergence rates are close to the one of low-order conforming
boundary element schemes. Numerical experiments confirm our conclusions.


\end{abstract}

\mysection{Introduction}{sec_intro}

This paper is about the approximation by radial basis functions of functions
in Sobolev spaces subject to a Dirichlet boundary condition. To the best of our
knowledge this is the first time that such functions are used and analyzed for
problems with essential boundary condition. There arise several difficulties
when imposing or analyzing trace conditions for spaces of radial basis functions (RBF):
\begin{itemize}
\item[(i)] Radial basis functions are selected by their center points
      on the domain of interest. Their supports are generally large and overlap
      with those of several other RBF.
      Considering boundary traces of RBF close to the boundary,
      their shapes and supports on the boundary vary continuously with the position
      of the center on the domain. Therefore, the structure and basis functions
      of trace spaces is not fixed, neither is there a fixed intrinsic basis on the
      boundary.
\item[(ii)] Analysis of stability and approximation properties of RBF and
      related operators is based on Fourier analysis in the so-called native space.
      Depending on the choice of RBF, this native space is a Sobolev space of
      high regularity. Considering domains with, e.g., Lipschitz boundary there is
      no obvious way to employ arguments from native spaces to traces.
\item[(iii)] Traditional arguments from finite element analysis
      (like locality and equivalence of norms in finite dimensional spaces)
      are difficult to apply to RBF for their very nature.
      That is why the native space is of central importance. In trace
      spaces of RBF, however, such arguments (locality, equivalence of norms) are even
      harder to come by due to the varying structure of functions, cf. (i).
\end{itemize}

In this paper we propose a mixed method employing RBF and finite elements (as Lagrangian
multipliers) to deal with essential boundary conditions.
We analyze non-conforming approximations with scaled radial basis functions in a
fractional order Sobolev space with homogeneous essential boundary condition.
The underlying model problem is the hypersingular integral equation on an open
surface for the solution of the Laplacian in the exterior domain
(with Neumann boundary condition). The energy space of this problem is a
Sobolev space of order $1/2$ on the surface with the condition that functions
can be extended by $0$.

The analysis comprises two principal problems. One is the approximation
theory for radial basis functions in fractional order Sobolev spaces with
essential boundary condition;
the other is the necessity of a non-conforming approach in a fractional order
Sobolev space.

Radial basis functions are well studied, mainly for the interpolation of
scattered data but also for the approximation of partial differential equations.
For some overviews see, e.g.,
\cite{Buhmann_03_RBF,Powell_92_TRB,Schaback_00_UTR,Wendland_05_SDA}.
In particular, Wendland studies the approximation for second order equations
with Neumann boundary condition \cite{Wendland_99_MGM}, and multiresolution
properties of scaled radial basis functions \cite{Wendland_10_MAS}.
Neumann boundary conditions allow
for extending the approximation analysis to the full space where
standard arguments apply (using the native space of the radial basis
functions). To the best of our knowledge there exists no analysis for boundary
value problems with essential boundary condition.
In this paper we tackle this problem for a Sobolev space of order $1/2$.
Let us note that spherical radial basis functions (on the closed sphere)
in this space are well analyzed,
see \cite{PhamT_08_NSB,PhamT_09_SPE,TranLGSS_09_BSR,TranLGSS_10_PPE}.
In this paper we consider the case of a general open smooth surface with the
particular problem of incorporating an essential boundary condition.

This boundary condition appears in a natural way when dealing with boundary
integral equations on open surfaces. In the case of a Neumann problem the
unknown of the integral equation is the jump across the surface
of the solution to the exterior problem \cite{Stephan_87_BIE}.
Since the jump of this solution vanishes
at the boundary of the surface the underlying energy space has to incorporate
this condition. However, there is no well-defined trace operator in the
corresponding Sobolev space $H^{1/2}$ so that this condition appears as part of the
norm; the corresponding space is denoted by $\tilde H^{1/2}$, sometimes also
referred to as $H^{1/2}_{00}$. Conforming approximations require that approximating
functions vanish at the boundary of the surface.
This causes no difficulty when using
piecewise polynomials; the corresponding method is called the boundary element method
(BEM). When using radial basis functions, however, a conforming and converging
method is difficult to construct since conformity requires that the centers of
the functions stay away from the boundary.
This requirement is unrealistic because in practice centers can stay
close to or even on the boundary.
We thus propose a non-conforming
approach where the essential boundary condition is implemented by a Lagrangian
multiplier. For the BEM such procedures have been studied, with resulting
almost quasi-optimal convergence, see \cite{GaticaHH_09_BLM,HealeyH_10_MBE,HeuerS_09_CRB}.
Here we use a different approach where the surface is extended to a larger surface
so that supports of radial basis functions remain inside the extended surface.
We then use a Lagrangian
multiplier on the extended part of the surface to make the approximating functions
vanish there. This idea is similar to a penalty or fictitious domain method.

An overview of the rest of the paper is as follows.
In the next section we recall Sobolev spaces, present our model problem
and give an equivalent mixed formulation which will be used for the
discretization with radial basis functions. In Section~\ref{sec_discrete} we
introduce scaled radial basis functions and the discrete scheme, and we
list our theoretical results which are proved in subsequent sections.
In Section~\ref{sec_cea} we prove the quasi-optimal convergence of our scheme,
and in Section~\ref{sec_rbf} we present an approximation theory for scaled radial
basis functions.
In Section~\ref{sec_con} we resume our theoretical results
and conclude that the resulting convergence
order tends to the one of a standard BEM when the regularity of the radial
basis functions grows. This conclusion is based on numerical evidence of boundedness
of an additional stability term that arises in our analysis.
Finally, in Section~\ref{sec_num} we report on
some numerical results that underline the convergence properties of our method.

Throughout the paper, the symbol $\lesssim$ will be used in the usual
sense. In short, $a(h,k,r)\lesssim b(h,k,r)$ when there exists a constant
$C>0$ independent of discretization or scaling parameters $h$, $k$, $r$
(except otherwise noted) such that $a(h,k,r) \le C b(h,k,r)$. The double
inequality $a \lesssim b \lesssim a$ is simplified to $a \simeq b$.

\mysection{Model problem and mixed formulation}{sec_model}

In order to introduce our model problem and its mixed formulation we need
to recall the definition of some Sobolev spaces.

We consider standard Sobolev spaces of integer order and
define fractional order spaces by interpolation, using the real K-method,
see \cite{BerghL_76_IS}.
For a Lipschitz domain $\Omega$ and $0<s<1$ we use the spaces
\[
   H^s(\Omega):=[L_2(\Omega), H^1(\Omega)]_s,\qquad
   \tilde H^s(\Omega):=[L_2(\Omega), H_0^1(\Omega)]_s
\]
where the norm in $H_0^1(\Omega)$ is the $H^1(\Omega)$-semi-norm.
For orders $s>1$, the spaces are defined by interpolation between $L_2(\Omega)$
and a correspondingly higher integer order Sobolev space.
For $s\in (0,1/2)$, $\|\cdot\|_{\tilde H^s(\Omega)}$ and
$\|\cdot\|_{H^s(\Omega)}$ are equivalent norms whereas for $s\in(1/2,1)$
there holds $\tilde H^s(\Omega) = H_0^s(\Omega)$,
the latter space consisting of functions whose traces on $\partial\Omega$
vanish. Generally, $\tilde H^s(\Omega)$ consists of functions which
are continuously extendable by zero onto a larger domain.
For $s>0$ the spaces $H^{-s}(\Omega)$ and $\tilde H^{-s}(\Omega)$ are the
dual spaces of $\tilde H^s(\Omega)$ and $H^s(\Omega)$, respectively.
Similarly we define Sobolev spaces on surfaces.

In the analysis we need some more norms.
In the literature different definitions of Sobolev norms are being
used and we have to be careful to check their equivalence when combining
different results. Apart from interpolation norms,
on a domain $\Sigma\subset\R^2$ ($\Sigma=\R^2$ is allowed),
we also need the Sobolev-Slobodeckij norm. For $s=m+\sigma$ with
integer $m\ge 0$ and $\sigma\in (0,1)$ we define
\[
   \|v\|_{s,\Sigma}
   :=
   \Bigl(\|v\|_{m,\Sigma}^2 + |v|_{s,\Sigma}^2\Bigr)^{1/2}
\]
with semi-norm
\[
   |v|_{s,\Sigma}
   =
   \Bigl(
   \sum_{|\alpha|=m}\int_\Sigma\int_\Sigma
   \frac {|D^\alpha v(x) - D^\alpha v(y)|^2}{|x-y|^{2+2\sigma}}\,dx\,dy
   \Bigr)^{1/2}
\]
and multi-index $\alpha=(\alpha_1,\alpha_2)$.
Here, $\|\cdot\|_{m,\Sigma}$ is the standard Sobolev norm, as before,
\[
   \|v\|_{m,\Sigma}^2 = \sum_{|\alpha|\le m} \|D^\alpha v\|_{L_2(\Sigma)}^2.
\]
Our model problem is as follows.
{\em For given $f\in H^{-1/2}(\G)$ find $\phi\in\tHp(\G)$ such that}
\begin{equation}\label{IE}
   W\phi(x):=-\frac 1{4\pi}\frac {\partial}{\partial \bn_x}
              \int_{\G} \phi(y) \frac {\partial}{\partial \bn_y} \frac 1{|x-y|}
              \,dS_y
   = f(x),\quad x\in\G.
\end{equation}
Here, $\G$ is a smooth open surface with piecewise smooth Lipschitz boundary
and $\bn$ is a normal unit vector on $\G$. We refer to \cite{Stephan_87_BIE}
for the setting of this model problem and the error analysis of its
conforming approximation by boundary elements.

Later we will extend $\G$ to a larger
surface and use the notation $W$ for the hypersingular operator on the
extended surface as well.

For ease of presentation we restrict ourselves to the case of a flat surface
and assume that $\G$
is a polygonal Lipschitz domain in $\R^2$. Then, in particular, $\G$ satisfies
an interior cone condition.

We intend to approximate the solution $\phi$ to (\ref{IE}) by radial basis
functions.
The conformity condition that approximation spaces be subspaces of
$\tHp(\G)$ requires that discrete functions vanish at the boundary
$\g$ of $\G$. This means that points for the definition of radial
basis functions could not be freely selected since one wants to take
radial basis functions of uniform radius which cannot be too small,
as is well known. In order to be able to freely choose center points we extend
the domain $\G$ to a larger domain $\tG$ so that, for a given parameter
$r>0$
there holds $\dist(\g, \tg)>r$. Here, $\tg$ denotes the boundary of $\tG$
and we assume that $\tG$ is at least Lipschitz.
Later, the parameter $r$ will be the scaling parameter for the radial basis
functions. Scaling the basis functions appropriately, we will be able to
select center points anywhere on $\bar\G$ and the discrete spaces of
radial basis functions will be subspaces of $\tHp(\tG)$.

A standard variational formulation of (\ref{IE}) is: {\em find
$\phi\in\tHp(\G)$ such that}
\begin{equation}\label{weak_std}
   \<W\phi,\psi\>_\G = \<f,\psi\>_\G\qquad\forall\psi\in\tHp(\G).
\end{equation}
However, to give the setting for the discrete method we consider a
non-standard formulation on the extended domain $\tG$:
{\em find $(\tphi,\lam)\in\tHp(\tG)\times\tHm_\g(\Gc)$ such that}
\begin{equation}\label{weak}
\ba{llll}
   \<W\tphi,\psi\>_\tG - \<\lam,\psi\>_\Gc &=& \<f^0,\psi\>_\tG
   &\quad \forall \psi\in\tHp(\tG),\\
   \<\mu,\tphi\>_\Gc                       &=& 0
   &\quad \forall \mu\in\tHm_\g(\Gc).
\ea
\end{equation}
Here, $\Gc:=\tG\setminus\bar\G$ with boundary $\g\cup\tg$ consisting
of two connected components ($\Gc$ is an annular domain).
Also, we use the following notation for the dual space:
\[
   \tHm_\g(\Gc)
   :=
   \left(\tHp_{\partial\Gc\setminus\g}(\Gc)\right)'
   =
   \left(\tHp_\tg(\Gc)\right)'
\]
where
\[
   \tHp_\tg(\Gc)
   :=
   [L_2(\Gc), H_{0,\tg}^1(\Gc)]_{1/2}
\]
with $H_{0,\tg}^1(\Gc)$ being the space of $H^1(\Gc)$-functions whose
traces on $\tg$ vanish. In particular, any $\psi\in \tHp_\tg(\Gc)$
is extendable by zero to a function
$\psi^0\in H^{1/2}(\R^2\setminus\bar\G)$.
Below, we will use this notation, instead of $\Gc$, also on extended
domains $\Gck$ depending on a mesh parameter $k$.

Note that $\tHp(\G)=\tHp_\g(\G)$ and
\[
   \tHm_\g(\G)=(\tHp_\emptyset(\G))'=(H^{1/2}(\G))'=\tHm(\G).
\]
Also, for $f$ defined on $\G$, $f^0$ denotes the
extension of $f$ by $0$ onto $\tG$ (we will use this generic notation
throughout for the extension to any domain which will be clear from
the particular situation).

There obviously holds the following result.

\begin{prop} \label{prop1}
Let $f\in H^{-1/2}(\G)$ so that $f^0\in H^{-1/2}(\tG)$. Then the
formulations \eqref{weak_std} and \eqref{weak} are equivalent and have
unique solutions. There holds
$\tphi|_\G=\phi$ and $\tphi|_{\tG\setminus\bar\G}=0$,
i.e. $\tphi=\phi^0$, and $\lam=W\phi^0$ on $\tG\setminus\bar\G$.
\end{prop}

\mysection{Discrete method and theoretical results}{sec_discrete}

We solve \eqref{weak} by approximating $\tphi$ by radial basis functions
and $\lam$ by piecewise constant functions. To this end let
$\Phi$ denote a non-negative radial basis function centered around $x=0$
with compact support $\bar B(0,1)$
($B(y,s)$ is the disc $\{x\in\R^2;\;|x-y|<s\}$)
and Fourier transform
\begin{equation}\label{hatPhi}
   \CF(\Phi)(\xi) = \hat\Phi(\xi)
   \simeq (1+|\xi|^2)^{-\tau},\quad \xi\in\R^2
\end{equation}
for $\tau>1$. The parameter $\tau$ is fixed throughout.
We consider scaled radial basis functions
\[
   \Phi_r(x):= r^{-2}\Phi(x/r),\quad r>0, \quad x\in\R^2,
\]
so that
\begin{equation}\label{hatPhir}
   \hat\Phi_r(\xi) \simeq (1+r^2|\xi|^2)^{-\tau}.
\end{equation}
Selecting a finite set of nodes $X:=\{x_1,\ldots,x_N\}\subset\bar\G$
we define the discrete space
\[
   H_{X,r} := \spn\{\phi_1,\ldots,\phi_N\}
   \quad\mbox{where}\quad
   \phi_i(x):=\Phi_r(x-x_i),\ i=1,\ldots,N.
\]
Since the nodes can be near to or even on the boundary of
$\Gamma$, the supports of the scaled radial basis functions
are not necessarily subsets of $\Gamma$. We extend
$\Gamma$ to a fixed larger domain
$\tG$ satisfying $\supp(\phi_i)\subset\tG$ for any $\phi_i\in
H_{X,r}$, any chosen set $X$ and any $r\in (0,r_0]$ where
$r_0>0$ is fixed.

We also need the {\em mesh norm} $h_{X,\Sigma}$ (for $\Sigma\subset\R^2$)
defined by
\[
   h_{X,\Sigma} := \sup_{x\in\Sigma}\dist(x,X).
\]

We extend $\G$ by a strip of shape-regular,
quasi-uniform quadrilateral elements $T$ of diameter
proportional to $k$ as indicated
in Figure~\ref{fig_ext}.
We require that the minimum diameter and length of the smallest edge on
$\g$ are not smaller than $k$. With shape-regularity we refer to elements
whose minimum (respectively, maximum) interior angles are
bounded from below (respectively, from above) by a positive
constant less than $\pi$.
We denote this mesh by $\CT_k$ and require that

\noindent{\bf $\CT_k$ is geometrically conforming with $\g$}:
each element $T\in\CT_k$ has either one or more entire edges in common
with $\g$, or $\bar T\cap\g$ is a vertex of $\G$ and of $T$. The latter
case happens at most once for each vertex of $\G$.

The extended domain is denoted by $\tGk$:
\[
   \tGk={\rm interior}(\bar\G\cup\{\bar T;\; T\in\CT_k\}).
\]
In the following, we choose the mesh size $k$ accordingly to the scaling
parameter $r$: $k>r$ and $k$ is small enough so that $\tGk\subset\tG$ and
no element touches two vertices of $\G$.
The assumption $k>r$ guarantees that the supports of the scaled
radial basis functions are within $\tGk$.

We introduce the notation $\Gck$ for the strip
$\tGk\setminus\bar\G$ and $\tgk:=\partial\tGk$ for the boundary of
$\tGk$.

Let us collect the assumptions we have made:
\begin{itemize}
\item[(A1)] The scaling parameter $r$ and mesh parameter $k$ are bounded,
            $r\in (0,r_0]$, $r<k<k_0$.
            $\CT_k$ is a quasi-uniform mesh of shape-regular quadrilaterals
            of diameter proportional to $k$ with minimum
            diameter not less than $k$ and
            with edges on $\g$ of length not less than $k$. Moreover, $\CT_k$
            is geometrically conforming with $\g$ and no element touches
            two vertices of $\G$.
\end{itemize}

\begin{figure}[htb]
\begin{center}
\includegraphics[width=0.5\textwidth]{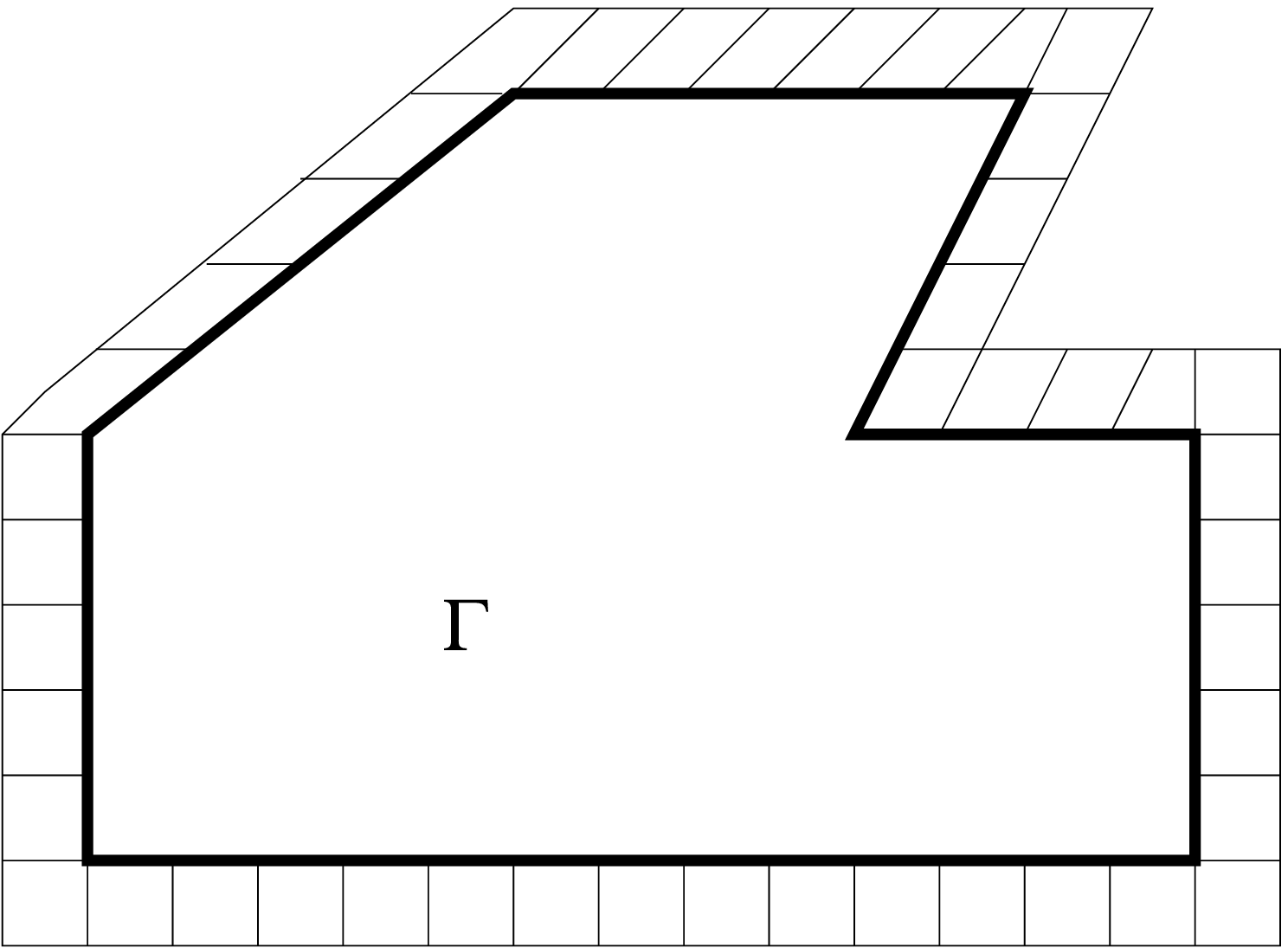}
\end{center}
\caption{Domain $\G$ extended by strip $\Gck$ with mesh $\CT_k$.}
\label{fig_ext}
\end{figure}

Now, for the approximation of the Lagrangian multiplier $\lam$ we
take piecewise constant functions,
\[
   M_k := \{\mu\in L_2(\Gck);\; \mu|_T=const\ \forall T\in\CT_k\}.
\]
Using these discrete spaces, the boundary element scheme with
radial basis functions and Lagrangian multiplier for the approximate
solution of (\ref{weak}) is:
{\em find $(\phi_N,\lam_k)\in H_{X,r}\times M_k$ such that}
\begin{equation}\label{BEM}
\ba{llll}
   \<W\phi_N,\psi\>_\tGk - \<\lam_k,\psi\>_\Gck &=& \<f^0,\psi\>_\tGk
   &\quad \forall \psi\in H_{X,r},\\
   \<\mu,\phi_N\>_\Gck                          &=& 0
   &\quad \forall \mu\in M_k.
\ea
\end{equation}
Defining the subspace
\[
   V_{X,r}
   :=
   \{\psi\in H_{X,r};\; \<\mu,\psi\>_\Gck = 0\quad\forall\mu\in M_k\},
\]
{and assuming an inf-sup condition for
the bilinear form $\<\cdot,\cdot\>_\Gck$ on $M_k\times H_{X,r}$,}
\eqref{BEM} is equivalent to: {\em find $\phi_N\in V_{X,r}$ such that}
\[
   \<W\phi_N,\psi\>_\tGk = \<f^0,\psi\>_\tGk
   \quad\forall \psi\in V_{X,r}.
\]

Apart from assumption (A1) we will need some more properties of $X$
for our analysis. We must be able to fix a constant function on any
given element by testing with a radial basis function.

To make this precise, let $\CC\subset\CT_k$ denote the set of
elements $T$ such that the closure $\overline T$ intersects $\g$
only at a vertex of $\G$.
For instance, in the example of Figure~\ref{fig_ext} there are five such
elements, at five of the six convex vertices.
Further we denote $\CE:=\CT_k\setminus\CC$. Then we assume:

\begin{itemize}
\item[(A2)]\quad
\(\displaystyle
   \forall T\in\CE\quad \exists x_i\in X:\quad
   \supp(\phi_i)\cap\Gck \subset T
\)
\item[(A3)]\quad
\(\displaystyle
   \forall T\in\CC:\quad T\mbox{ has exactly two neighbors}\ T_1,T_2\in\CE,
\)\\
\hspace*{2.5em}
\(\displaystyle
   \qquad\quad \exists x_i\in X:\quad
   \supp(\phi_i)\cap T \not=\emptyset,\quad
   \supp(\phi_i)\cap \tilde T =\emptyset\
   \forall\tilde T\in\CT_k\setminus\{T,T_1,T_2\}
\)
\end{itemize}
For any $T\in\CT_k$ there may be more than one $x_i\in X$
satisfying either (A2) or (A3). We will denote one of these
points $x_i$ by $x_{i(T)}$.

Moreover, we assume that there is {\em substantial} overlap between
$H_{X,r}$ and $M_k$:
\begin{itemize}
\item[(A4)]\quad
\(\displaystyle
   \exists\kappa>0:\quad
   \meas(\supp(\phi_{i(T)})\cap T)
   \ge
   \kappa\,\meas(\supp(\phi_{i(T)}))
   =\kappa\, \pi r^2\quad
   \forall T\in\CT_k
\)
\end{itemize}

\begin{remark}
Note that the discrete scheme is defined on the domain $\tGk$ which depends
on $k$. But this causes no difficulty since, under the assumptions made,
all the domains $\tGk$ can be extended to the fixed domain $\tG$ and the
spaces involved allow for extension by zero of their elements to $\tG$.
More precisely, there holds
\[
   H_{X,r}\subset \tHp(\tGk)\subset \tHp(\tG),\quad
   M_k    \subset \tHm_\g(\Gck) \subset \tHm_\g(\Gc)
\]
where the latter inclusions are to be understood as the uniformly continuous
injections of the respective extension by zero.
\end{remark}

\paragraph{Theoretical results.}
Based on assumptions {\rm (A1)--(A4)} we prove the following results.
\begin{itemize}
\item The discrete scheme (\ref{BEM}) converges quasi-optimally;
      see Theorem~\ref{thm_cea1}.
\item The error of the best approximation by radial basis functions
      in the constrained space $V_{X,r}$ (mean value zero on elements of the extension)
      can be bounded by the error of the best approximation in the
      unconstrained space $H_{X,r}$ on $\G$ plus a stability term;
      see Theorem~\ref{thm_cea}.
\item We prove an error estimate for the best approximation by scaled
      radial basis functions (Theorem~\ref{thm_approx}) which appears to be
      sharp according to our numerical results in Section~\ref{sec_num}.
\end{itemize}

We are unable to show an appropriate bound for the stability term
$|\psi|_{H^t(\Gck)}$ in Theorem~\rref{thm_cea}. Based on our numerical tests, however,
we conjecture that this term is appropriately bounded so that we conclude
for the choice $k\simeq r\simeq h_{X,\G}^{1-t/\tau}$ the overall error estimate
\begin{equation} \label{main}
   \|\phi-\phi_N\|_{H^{1/2}(\G)}
   \le
   C\, h_{X,\G}^{(t-1/2-\epsilon)(1-t/\tau)}
   \quad \epsilon>0, t\in(1/2,1].
\end{equation}
The constant $C$ would depend on $t$, $\epsilon$ and $\tau$
but not on $X$ and $h_{X,\G}$ under the assumptions made.
We refer to Section~\ref{sec_con}, in particular \eqref{main_con}, for more details.

\begin{remark}
In our case of an open surface $\G$, the solution $\phi$ of \eqref{IE}
has strong singularities along $\gamma$. This limits the convergence order of
approximation schemes. Measuring regularity in standard Sobolev spaces,
the $h$-version of the standard boundary element method with quasi-uniform
meshes (and mesh size $h$) converges like
\begin{equation}\label{est_std}
   \|\phi-\phi_h\|_{\tHp(\G)} \lesssim h^{t-1/2} \|\phi\|_{H^t(\G)},
\end{equation}
cf. \rcite{vonPetersdorffS_90_DEC}. Since $\phi\in H^t(\G)$ for any
$t<1$ but $\phi\not\in H^1(\G)$ in general there is an upper limit $1/2$ for
the convergence order in $h$. An optimal error estimate making use of the
type of appearing singularities is
\[
   \|\phi-\phi_h\|_{\tilde H^{1/2}(\G)}
   \lesssim h^{1/2},
\]
again for quasi-uniform meshes, see \rcite{BespalovH_08_hpB}.
In our case of radial basis functions the mesh size $h$ corresponds to the
mesh norm $h_{X,\G}$ and for quasi-uniform distribution of nodes this
parameter is equivalent to $h$ for quasi-uniform meshes. The estimate
\eqref{main} exactly reflects the error estimate
\eqref{est_std} for large $\tau$. Selecting sufficiently smooth
radial basis functions, which corresponds to large $\tau$, one gets as close
as wanted to the convergence order $1/2$.
{The use of lower regularity of radial basis functions results in a lower
convergence order.}
Our numerical experiments reported below confirm the predicted
influence of $\tau$.
\end{remark}

\mysection{Quasi-optimal convergence}{sec_cea}

We prove the quasi-optimal convergence of \eqref{BEM} for
the approximation of $\phi$. Later, in Section~\ref{sec_rbf}
we study the approximation problems
for $\phi$ and $\lam$ to derive convergence orders. We do not bound
the Galerkin error for the Lagrangian multiplier $\lam$ since, on the
one hand, proving a discrete inf-sup condition for the bilinear forms
$\<\cdot,\cdot\>_\Gck$ is an open problem and, on the other hand, the function
$\lam$ is of no physical interest. We will therefore analyze \eqref{BEM} without
using a discrete inf-sup condition.

\begin{theorem} \label{thm_cea1}
Let the assumptions {\rm (A1)--(A3)} be satisfied.
Then there exists a unique solution to \eqref{BEM} and there
holds the quasi-optimal error estimate
\[
   \|\phi^0-\phi_N\|_{\tHp(\tGk)}
   \lesssim
   \inf_{\psi\in V_{X,r}} \|\phi^0-\psi\|_{\tHp(\tGk)}
   +
   \inf_{\mu\in M_k} \|\lam-\mu\|_{\tHm_\g(\Gck)}.
\]
\end{theorem}

\begin{proof}
The existence and uniqueness of $(\phi_N,\lam_k)\in H_{X,r}\times M_k$
follows from the {Babu\v ska}-Brezzi theory. Specifically,
{uniformly in $k$},
\begin{itemize}
\item[(i)] $\<W\cdot,\cdot\>_\tGk$ is bounded:
\begin{equation}\label{Wbound}
   \<Wv,\psi\>_\tGk = \<Wv^0,\psi^0\>_\tG
   \lesssim
   \|v^0\|_{\tHp(\tG)} \|\psi^0\|_{\tHp(\tG)}
   \simeq
   \|v\|_{\tHp(\tGk)} \|\psi\|_{\tHp(\tGk)}
\end{equation}
for any $v, \psi\in\tHp(\tGk)$
\item[(ii)] $\<W\cdot,\cdot\>_\tGk$ is elliptic:
\begin{equation}\label{ell}
   \<W\psi,\psi\>_\tGk = \<W\psi^0,\psi^0\>_\tG
   \gtrsim
   \|\psi^0\|_{\tHp(\tG)}^2
   \simeq
   \|\psi\|_{\tHp(\tGk)}^2
   \quad\forall\psi\in\tHp(\tGk)
\end{equation}
\item[(iii)] $\<\cdot,\cdot\>_\Gck$ is bounded:
\begin{equation}\label{bound}
   \<\mu, \psi\>_\Gck
   \le
   \|\mu\|_{\tHm_\g(\Gck)} \|\psi\|_{\tHp_\tgk(\Gck)}
   \le
   \|\mu\|_{\tHm_\g(\Gck)} \|\psi\|_{\tHp(\tGk)}
\end{equation}
for any $\mu\in\tHm_\g(\Gck)$ and $\psi\in\tHp(\tGk)$
\item[(iv)] The linear form defined by $f^0$ is bounded, i.e.
\[
   \<f^0, \psi\>_\tGk
   \le
   \|f^0\|_{H^{-1/2}(\tGk)} \|\psi\|_{\tHp(\tGk)}
   \lesssim
   \|f^0\|_{H^{-1/2}(\tG)} \|\psi\|_{\tHp(\tGk)}
   \quad\forall \psi\in\tHp(\tGk).
\]
\end{itemize}
Here, we used several uniform norm-equivalences, e.g.
$\|v\|_{\tHp(\tGk)}\simeq\|v^0\|_{\tHp(\tG)}$
since $\tHp(\tGk)$ (resp., $\tHp(\tG)$) is defined by interpolation
between $L_2(\tGk)$ and $H_0^1(\tGk)$ (resp, between
$L_2(\tG)$ and $H_0^1(\tG)$), and
\[
   \|v\|_{H_0^1(\tGk)}
   =
   |v|_{H^1(\tGk)}
   =
   |v^0|_{H^1(\tG)}
   =
   \|v^0\|_{H_0^1(\tG)}
   \qquad \forall v\in H_0^1(\tGk).
\]
Rather than proving a discrete inf-sup condition for the bilinear form
$\<\cdot,\cdot\>_\Gck$ we only show injectivity. More precisely,
\begin{itemize}
\item[(v)] $\<\cdot,\cdot\>_\Gck$ is discrete injective:
\begin{equation}\label{inj}
   \Bigl(\mu\in M_k:\quad \<\mu, \psi\>_\Gck = 0
         \quad\forall\psi\in H_{X,r}\Bigr)
   \qquad\Rightarrow\qquad \mu=0
\end{equation}
To see this we proceed as follows. Let $\mu=\sum_{T\in\CT_k}c_T\chi_T$
(with characteristic function $\chi_T$ on element $T$) satisfy \eqref{inj}.
For an element $T\in\CE$, i.e. $T$ has at least an entire edge in
common with $\g$, there exists, due to assumption (A2), a basis function
$\phi_i\in H_{X,r}$ whose support overlaps only with the element $T$.
Then $\<\mu, \phi_i\>_\Gck=c_T\<\chi_T, \phi_i\>_\Gck=0$ so that $c_T=0$.
Therefore, $\mu$ vanishes on all those elements. Elements not having an
entire edge in common with $\g$ can only be at convex vertices of $\G$
and are isolated. By assumption (A3) we can again choose a basis function
for each of those elements $T$, this time with the only condition that there
is some overlap between $T$ and the support of the corresponding basis
functions. Since we already know that $\mu$ vanishes on the neighboring
elements, the argument from before implies that $\mu$ vanishes also on
the remaining vertex elements. This proves \eqref{inj}.
\end{itemize}

The {Babu\v ska}-Brezzi theory then implies that there exists a unique solution
$(\phi_N,\lam_k)$ of \eqref{BEM}.

To prove the quasi-optimal convergence we first derive a Strang-type
error estimate. Following the standard procedure,
we use the triangle inequality and the uniform ellipticity \eqref{ell}
to conclude that for any $\psi\in V_{X,r}$ there holds
\begin{align*}
\lefteqn{
   \|\phi^0-\phi_N\|_{\tHp(\tGk)}
   \le
   \|\phi^0-\psi\|_{\tHp(\tGk)} + \|\phi_N-\psi\|_{\tHp(\tGk)}
}
   \\
   &\lesssim
   \|\phi^0-\psi\|_{\tHp(\tGk)}
   +
   \sup_{\varphi\in V_{X,r}\setminus\{0\}}
   \frac {\<W(\phi_N-\psi), \varphi\>_\tGk}
         {\|\varphi\|_{\tHp(\tGk)}}
   \nonumber
   \\
   &\lesssim
   \|\phi^0-\psi\|_{\tHp(\tGk)}
   +
   \\
   &\qquad
   \sup_{\varphi\in V_{X,r}\setminus\{0\}}
   \frac {\<W(\phi^0-\psi), \varphi\>_\tGk}
         {\|\varphi\|_{\tHp(\tGk)}}
   +
   \sup_{\varphi\in V_{X,r}\setminus\{0\}}
   \frac {\<W(\phi^0-\phi_N), \varphi\>_\tGk}
         {\|\varphi\|_{\tHp(\tGk)}}.
\end{align*}
Then, using the uniform boundedness \eqref{Wbound}, we find
\begin{equation}\label{strang}
   \|\phi^0-\phi_N\|_{\tHp(\tGk)}
   \lesssim
   \inf_{\psi\in V_{X,r}} \|\phi^0-\psi\|_{\tHp(\tGk)}
   +
   \sup_{\varphi\in V_{X,r}\setminus\{0\}}
   \frac {\<W(\phi^0-\phi_N), \varphi\>_\tGk}
         {\|\varphi\|_{\tHp(\tGk)}}.
\end{equation}
This is the Strang-type error estimate for a non-conforming approximation.
For a conforming method the latter term above vanishes.

It remains to bound the non-conformity term.
Combining \eqref{weak} and \eqref{BEM} one finds that there holds
(note that $\tilde\phi=\phi^0$ by Proposition~\ref{prop1})
\[
   \<W(\phi^0-\phi_N), \varphi\>_\tGk
   =
   \<\lam-\lam_k, \varphi\>_\Gck
   =
   \<\lam-\mu, \varphi\>_\Gck
   \quad\forall \varphi\in V_{X,r},\ \forall\mu\in M_k.
\]
Combination with \eqref{strang} and application of \eqref{bound}
yields
\[
   \|\phi^0-\phi_N\|_{\tHp(\tGk)}
   \lesssim
   \inf_{\psi\in V_{X,r}} \|\phi^0-\psi\|_{\tHp(\tGk)}
   +
   \inf_{\mu\in M_k} \|\lam-\mu\|_{\tHm_\g(\Gck)}.
\]
This finishes the proof.
\end{proof}

In the next step we analyze the approximation error in the
constrained space.

\begin{lemma} \label{la_approx}
Let the assumptions {\rm(A1)--(A4)} be satisfied and let $s,t\in (1/2,1]$
with $s<t$. Then for any $\phi\in\tilde H^t(\G)$ there holds
\[
   \inf_{\psi\in V_{X,r}} \|\phi^0-\psi\|_{\tilde H^s(\tGk)}
   \le
   C(s,t)
   \inf_{\psi\in H_{X,r}}
   \Bigl(
      \|\phi-\psi\|_{H^s(\G)}
      +
      \delta(s,t,k,r) |\psi|_{H^t(\Gck)}
   \Bigr)
\]
with
\[
   \delta(s,t,k,r)
   :=
   k^{1+t} r^{-(1+s)} + k^{t-s} \bigl( 1+ k^{1+t} r^{-(1+t)} \bigr).
\]
\end{lemma}

\begin{proof}
Let $\psi_1\in H_{X,r}$ be a minimizer of
\[
   \|\phi-\psi\|_{H^s(\G)}
   +
   \delta(s,t,k,r) |\psi|_{H^t(\Gck)}
\]
among the elements $\psi\in H_{X,r}$.
We will construct a function $\psi_2\in H_{X,r}$ such that
\begin{equation}\label{app1}
   \<\mu, \psi_1-\psi_2\>_\Gck = 0\quad\forall \mu\in M_k,
\end{equation}
that is $\psi:=\psi_1-\psi_2\in V_{X,r}$, and such that
\begin{equation}\label{app2}
   \|\phi^0-\psi\|_{\tilde H^s(\tGk)}
   \le
   C(s,t)
   \Bigl(
   \|\phi-\psi_1\|_{H^s(\G)}
   +
   \delta(s,t,k,r) |\psi_1|_{H^t(\Gck)}
   \Bigr).
\end{equation}
To this end recall the notation of $\CC$ for elements
that touch $\g$ only at a vertex and of $\CE$ for elements touching $\g$
with at least an edge.
Let us denote by $\CE_0\subset\CE$ the set of elements
which do not touch any element of $\CC$.

By assumption, the radial basis functions are
non-negative. We consider the normalized functions
$\phi_i^*:=\phi_i/\|\phi_i\|_{L_1(T)}$ for
$i=i(T)$ so that $\|\phi_i^*\|_{L_1(T)}=1$.
This normalization is well defined
since there is substantial overlap between $\supp(\phi_{i(T)})$ and $T$
by assumption (A4).
We make the ansatz
\[
   \psi_2 = \sum_{T\in\CT_k} c_T \phi^*_{i(T)}.
\]
For $T\in\CE_0$ we define
\[
   c_T=\int_T\psi_1
\]
so that $\<1, \psi_1-\psi_2\>_T = 0$ due to the chosen normalization.

If $\CC\not=\emptyset$ then there remain elements in $\CE$
which are associated with vertices of $\G$.
Let us pick one vertex
where this happens, i.e. there is an element $T_2\in\CC$ touching $\g$
at this vertex and there are two neighboring elements
$T_1, T_3\in\CE\setminus\CE_0$. For illustration see Figure~\ref{fig_corner}.

\begin{figure}[htb]
\begin{center}
\includegraphics[width=0.4\textwidth]{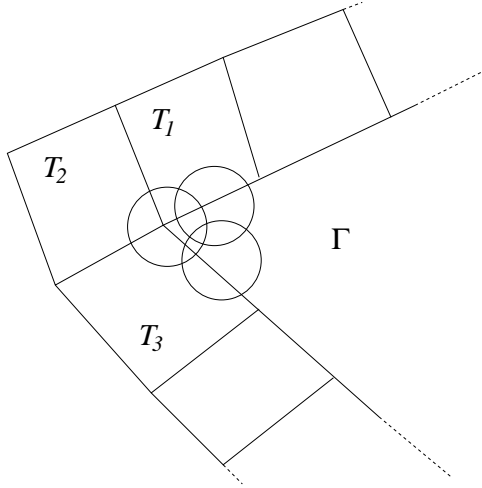}
\end{center}
\caption{Corner elements and supports of associated radial basis functions.}
\label{fig_corner}
\end{figure}

For this vertex we define
\[
   c_{T_1}=\int_{T_1}\psi_1 - \int_{T_1}\phi^*_{i(T_2)}\int_{T_2}\psi_1,\quad
   c_{T_2}=\int_{T_2}\psi_1,\quad
   c_{T_3}=\int_{T_3}\psi_1 - \int_{T_3}\phi^*_{i(T_2)}\int_{T_2}\psi_1.
\]
Repeating this construction for all elements $T\in\CC$ we obtain
$\psi_2\in H_{X,r}$ satisfying \eqref{app1}.

It remains to verify \eqref{app2}.
Since $s\in (1/2,1]$ there holds by \cite[Lemma 3.1]{BespalovH_05_pB}
\[
   \|\phi^0-\psi\|_{\tilde H^s(\tGk)}^{2}
   \simeq
   \|\phi^0-\psi\|_{H^s(\tGk)}^{2}
   \simeq
   \|\phi^0-\psi\|_{s,\tGk}^{2}
   \leq
   2\|\phi^0-\psi\|_{s,\G}^{2}
   +
   2\|\phi^0-\psi\|_{s,\G_k^c}^{2}.
\]
Here, we used the equivalence of the interpolation norm
$\|\cdot\|_{H^s(\tGk)}$ and the Slobodeckij norm
$\|\cdot\|_{s,\tGk}$ with constants independent of $k$ due to
$\meas(\G) \le \meas(\tGk) \le 2\meas(\G)$. However, they may depend
on $s$ and may be unbounded when $s\to 1/2$.
By noting that $\phi^0=\phi$ on $\G$ and $\phi^0=0$ on $\Gck$, and by
using \cite[Lemma 3.5]{BespalovH_08_hpB} we deduce
\[
   \|\phi^0-\psi\|_{\tilde H^s(\tGk)}^{2}
   \lesssim
   \|\phi-\psi\|_{s,\G}^{2}
   +
   \sum_{T\in\CT_k}
   \Bigl( k^{-2s} \|\psi\|_{L_2(T)}^2 + |\psi|_{s,T}^2 \Bigr).
\]
Using the triangle inequality with $\psi=\psi_1-\psi_2$
and the equivalence of norms again, this time on $\G$, we obtain
\begin{align*}
   \|\phi^0-\psi\|_{\tilde H^s(\tGk)}^2
   &\lesssim
   \|\phi-\psi_1\|_{H^s(\G)}^2
   +
   \|\psi_2\|_{H^s(\G)}^2
   +
   \sum_{T\in\CT_k}
   \Bigl( k^{-2s} \|\psi\|_{L_2(T)}^2 + |\psi|_{s,T}^2 \Bigr).
\end{align*}
Now, the support of $\psi_2$ is confined to a neighborhood of the
boundary of $\G$. Therefore, by a Poincar\'e-Friedrichs inequality
its norms can be replaced by the semi-norm, giving
\[
   \|\phi^0-\psi\|_{\tilde H^s(\tGk)}^2
   \lesssim
   \|\phi-\psi_1\|_{H^s(\G)}^2
   +
   |\psi_2|_{H^s(\G)}^2
   +
   \sum_{T\in\CT_k}
   \Bigl( k^{-2s} \|\psi\|_{L_2(T)}^2 + |\psi|_{s,T}^2 \Bigr).
\]
We finish the proof of the lemma by showing that
\begin{equation} \label{bound_psi2}
   |\psi_2|_{H^s(\G)} \lesssim k^{1+t} r^{-(1+s)} |\psi_1|_{H^t(\Gck)}
\end{equation}
and
\begin{equation} \label{bound_psi}
   \sum_{T\in\CT_k}
   \Bigl( k^{-2s} \|\psi\|_{L_2(T)}^2 + |\psi|_{s,T}^2 \Bigr)
   \lesssim
   k^{2(t-s)} \bigl( 1+ k^{1+t} r^{-(1+t)} \bigr)^{2}
   |\psi_1|_{H^t(\Gck)}^{2}.
\end{equation}
{\bf Proof of \eqref{bound_psi2}}.
By a coloring argument and the Cauchy-Schwarz inequality we start bounding
\begin{equation}\label{app4}
   |\psi_2|_{H^s(\G)}^2
   \lesssim
   \sum_{T\in\CT_k} c_T^2|\phi^*_{i(T)}|_{H^s(\supp(\phi^*_{i(T)}))}^2.
\end{equation}
Here we used that only a fixed number (independent of all relevant parameters)
of appearing radial basis functions overlap.
By the scaling property of the $H^s$-semi-norm (see, e.g., \cite{Heuer_01_ApS})
there holds
\[
   |\phi_{i(T)}|_{H^s(\supp(\phi_{i(T)}))}^2
   \simeq
   r^{2-2s} |\phi_{i(T)}|_{L_\infty(\supp(\phi_{i(T)}))}^2
   \simeq
   r^{2-2s} r^{-4}
\]
and by the assumption of substantial overlap (A4) one finds
\[
   \|\phi_{i(T)}\|_{L_1(T)}
   \simeq
   1.
\]
This proves
\begin{align} \label{app5}
   |\phi^*_{i(T)}|_{H^s(\supp(\phi^*_{i(T)}))}^2
   =
   \|\phi_{i(T)}\|_{L_1(T)}^{-2}
   |\phi_{i(T)}|_{H^s(\supp(\phi^*_{i(T)}))}^2
   \simeq
   r^{-2(1+s)}.
\end{align}
Now, for $T\in\CE_0$, again using scaling properties, transforming
to a reference element $\hat T$, denoting the transformed function
by adding the symbol ``$\,\hat{\ }\,$'',
and applying a Poincar\'e-Friedrichs inequality, we obtain
\begin{equation}\label{app6}
   c_T^2
   = (\int_T\psi_1)^2
   \simeq k^4 (\int_{\hat T} \hat\psi_1)^2
   \lesssim k^4 |\hat\psi_1|_{H^t(\hat T)}^2
   \simeq k^{2+2t} |\psi_1|_{H^t(T)}^2.
\end{equation}
For a corner element $T_2\in\CC$ with neighboring elements $T_1, T_3\in\CE$
we obtain
\begin{equation}\label{app7}
   c_{T_2}^2 \lesssim  k^{2+2t} |\psi_1|_{H^t(T_2)}^2
\end{equation}
as before and
\begin{equation}\label{app8}
   c_{T_1}^2
   \le
   2 \Bigl(\int_{T_1}\psi_1\Bigr)^2
   +
   2 \Bigl(\int_{T_1}\phi^*_{i(T_2)}\int_{T_2}\psi_1\Bigr)^2
   \lesssim
   k^{2+2t} |\psi_1|_{H^t(T_1)}^2
   +
   k^{2+2t} |\psi_1|_{H^t(T_2)}^2.
\end{equation}
In the last step we used that $\int_{T_1}\phi^*_{i(T_2)}\lesssim 1$
by the quasi-uniformity of the mesh, the normalization
$\int_{T_2}\phi^*_{i(T_2)}=1$ and the substantial overlap of
$\supp(\phi^*_{i(T_2)})$ with $T_2$.
Accordingly one bounds
\begin{equation}\label{app9}
   c_{T_3}^2
   \lesssim
   k^{2+2t} |\psi_1|_{H^t(T_3)}^2
   +
   k^{2+2t} |\psi_1|_{H^t(T_2)}^2
\end{equation}
and repeats this procedure for all edges where necessary.
Combining \eqref{app5}--\eqref{app9} and recalling \eqref{app4}
we obtain \eqref{bound_psi2}.

{\bf Proof of \eqref{bound_psi}.}
We use scaling arguments and a Poincar\'e-Friedrichs inequality as before.
By the integral-mean zero condition (or using that $\psi$ satisfies
a homogeneous boundary condition) and scaling properties we can bound
\[
   \|\psi\|_{L_2(T)}^2
   \lesssim k^{2t} |\psi|_{t,T}^2
   \lesssim k^{2} |\hat \psi|_{t,\hat T}^2
   \simeq k^{2} |\hat \psi|_{H^t(\hat T)}^2
   \lesssim k^{2t} |\psi|_{H^t(T)}^2
\]
and
\[
   |\psi|_{s,T}^2
   \lesssim
   k^{2(t-s)} |\psi|_{t,T}^2
   \lesssim
   k^{2-2s} |\hat \psi|_{t,\hat T}^2
   \simeq
   k^{2-2s} |\hat \psi|_{H^t(\hat T)}^2
   \lesssim
   k^{2(t-s)} |\psi|_{H^t(T)}^2
\]
so that
\begin{align*}
   \sum_{T\in\CT_k}
   \Bigl( k^{-2s} \|\psi\|_{L_2(T)}^2 + |\psi|_{s,T}^2 \Bigr)
   &\lesssim
   \sum_{T\in\CT_k}
   k^{2(t-s)} |\psi|_{H^t(T)}^2
   \lesssim
   k^{2(t-s)} |\psi|_{H^t(\Gck)}^2.
\end{align*}
Analogously to \eqref{bound_psi2} we can bound
\[
   |\psi_2|_{H^s(\Gck)} \lesssim
   k^{1+t} r^{-(1+t)} |\psi_1|_{H^t(\Gck)}.
\]
Therefore, the representation $\psi=\psi_1-\psi_2$ and the triangle
inequality, together with the previous estimate, yield
\begin{align*}
   \sum_{T\in\CT_k}
   \Bigl( k^{-2s} \|\psi\|_{L_2(T)}^2 + |\psi|_{s,T}^2 \Bigr)
   &\lesssim
   k^{2(t-s)} \Bigl( |\psi_1|_{H^t(\Gck)} +
|\psi_2|_{H^t(\Gck)} \Bigr)^{2}
   \\
   &\lesssim
   k^{2(t-s)} \Bigl( |\psi_1|_{H^t(\Gck)}
                     +
                     k^{1+t} r^{-(1+t)}
|\psi_1|_{H^t(\Gck)} \Bigr)^{2}.
\end{align*}
This finishes the proof of \eqref{bound_psi}.
\end{proof}

We can now state the main result of this section,
the quasi-optimal convergence of our scheme \eqref{BEM} in the unconstrained
space $H_{X,r}\times M_k$.

\begin{theorem} \label{thm_cea}
Let the assumptions {\rm(A1)--(A4)} be satisfied and let $t\in (1/2,1]$.
There exists a unique solution $\phi$ to \eqref{BEM}.
If $\phi\in\tilde H^t(\G)$ then there holds for $\epsilon>0$
the error estimate
\[
   \|\phi^0-\phi_N\|_{\tHp(\tGk)}
   \le
   C(\epsilon)
   \inf_{\psi\in H_{X,r}}
   \Bigl(
      \|\phi-\psi\|_{H^{1/2+\epsilon}(\G)}
      +
      \delta(\epsilon,t,k,r)
      |\psi|_{H^t(\Gck)}
   \Bigr)
   +
   \inf_{\mu\in M_k} \|\lam-\mu\|_{\tHm_\g(\Gck)}
\]
where $C(\epsilon)>0$ depends only on $\epsilon$ and with
\[
   \delta(\epsilon,t,k,r)
   :=
   k^{1+t} r^{-3/2-\epsilon}
   +
   k^{t-1/2-\epsilon} \bigl( 1+ k^{1+t} r^{-(1+t)} \bigr).
\]
\end{theorem}

\begin{proof}
The statement is a combination of Theorem~\ref{thm_cea1} and
Lemma~\ref{la_approx}, setting $s=1/2+\epsilon$ in the latter one.
\end{proof}

In the next section an error estimate for the best approximation
by scaled radial basis functions will be derived.

\mysection{Approximation by scaled radial basis functions}{sec_rbf}

Standard approximation theory of radial basis functions centers around
the native space defined by
\[
   \CN_\Phi
   :=
   \Bigl\{v\in L_2(\R^2);\;
     \int_{\fR^2} \frac {|\hat v(\xi)|^2}{\hat\Phi(\xi)}\,d\xi < \infty
   \Bigr\}
   \quad\mbox{with norm}\quad
   \|v\|_{\CN_\Phi}
   :=
   \|\hat\Phi^{-1/2}\hat v\|_{L_2(\fR^2)}
\]
which, for $\Phi$ satisfying \eqref{hatPhi},
is identical to the Sobolev space $H^\tau(\R^2)$ with equivalent norms
when defining
\[
   \|v\|_{H_F^s(\fR^2)}
   := \|(1+|\cdot|^2)^{s/2} \ \hat v(\cdot)\|_{L_2(\fR^2)} \qquad (s\in\R).
\]
It is well known that the interpolation operator $I_X$ defined by
\[
   I_X v \in H_X=\spn\{\Phi(\cdot-x_i);\; i=1,\ldots,N\};\qquad
   I_Xv(x_i) = v(x_i),\quad i=1,\ldots,N
\]
for $v\in\CN_\Phi$, satisfies
\[
   \|v-I_Xv\|_{\CN_\Phi} = \min_{\psi\in H_X} \|v-\psi\|_{\CN_\Phi},
\]
see \cite{Duchon_78_SEI}.

Using the scaled radial basis function $\Phi_r$ the native space
$\CN_{\Phi_r}$ is still identical to $H^\tau(\R^2)$ but its norm is
uniformly equivalent to the norm
\[
   \|v\|_{H_{F,r}^s(\fR^2)}
   :=
   \|(1+r^2|\cdot|^2)^{s/2} \ \hat v(\cdot)\|_{L_2(\fR^2)}
   \quad (r>0)
   \qquad\mbox{for}\quad s=\tau,
\]
cf. \eqref{hatPhir}.

For the analysis, we also need the scaled versions of the
norms defined in Section~\ref{sec_model}.
{For $s>0$, $s=m+\sigma$ with integer $m\ge 0$ and
{$\sigma\in(0,1)$},}
the scaled Sobolev-Slobodeckij norm is defined by
\[
   \|v\|_{s,r,\Sigma}
   :=
   \Bigl(
   \|v\|_{m,r,\Sigma}^2
   + r^{2s}|v|_{s,\Sigma}^2\Bigr)^{1/2}
   \qquad (r>0)
\]
where $\Sigma$ is a domain in $\R^2$
and
\[
   \|v\|_{m,r,\Sigma}^2
   :=
   \sum_{|\alpha|\le m} r^{2|\alpha|}
   \|D^\alpha v\|_{L_2(\Sigma)}^2
\]
with multi-index $\alpha=(\alpha_1,\alpha_2)$.
We also define the scaled interpolation spaces
\[
   H_r^s(\Sigma) = [H_r^m(\Sigma), H_r^{m+1}(\Sigma)]_{\sigma},
   \quad (s=m+\sigma \mbox{ as before})
\]
{with norm denoted by $\|\cdot\|_{H_r^s(\Sigma)}$.}
By the equivalence of the semi-norms $|v|_{s,\fR^2}$ and
$\||\cdot|^s\hat v(\cdot)\|_{L_2(\fR^2)}$
(cf.~\cite[Lemma~3.15]{McLean_00_SES}) and by repeating the
arguments of Theorem~B.7 in \cite{McLean_00_SES} we deduce
that
\begin{equation}\label{scaled s}
   \|v\|_{s,r,\fR^2}
   \simeq
   \|v\|_{H_{F,r}^s(\fR^2)}
   \simeq
   \|v\|_{H_{r}^s(\fR^2)}
   \quad \forall v\in {H^s(\R^2)},
   \quad r>0,
\end{equation}
where the constants are independent of $r$.
In particular, when $s=\tau$ there holds
\begin{equation}\label{scaled}
   \|v\|_{\tau,r,\fR^2} \simeq \|v\|_{H_{r}^\tau(\fR^2)}
                        \simeq \|v\|_{\CN_{\Phi_r}}
   \quad \forall v\in {H^\tau(\R^2)},
   \quad r>0.
\end{equation}
Furthermore, we need the following result.

\begin{lemma} \label{la_norms}
For any $s>0$ there exists a bounded extension operator
\begin{equation}\label{Es}
   E:\; H_r^s(\G)\to H_r^s(\R^2).
\end{equation}
As a consequence,there holds
\begin{equation}\label{equiv}
   \|v\|_{s,r,\G}
   \lesssim
   \|v\|_{H_r^s(\G)}
   \qquad\forall v\in {H^s(\G)}.
\end{equation}
In both cases the boundedness is uniform for $r>0$ bounded from above.
\end{lemma}

\begin{proof}
%
By Stein (see \cite[Section~3, Chapter~VI]{Stein_70_SID})
there is a bounded extension operator $E : H^m(\G) \to
H^m(\R^2)$ defined for all non-negative intergers $m$.
It follows that $E : H_r^m(\G) \to H_r^m(\R^2)$. Indeed, 
{for any integer $m$ there hold}
\begin{align*}
   \|Ev\|_{H_r^m(\fR^2)}^2
   &=
   \sum_{|\alpha|\le m}
   r^{2|\alpha|} \|D^\alpha(Ev)\|_{L_2(\fR^2)}^2
   \le
   \sum_{|\alpha|\le m}
   r^{2|\alpha|} \|Ev\|_{H^{|\alpha|}(\fR^2)}^2 \\
   &\lesssim
   \sum_{|\alpha|\le m}
   r^{2|\alpha|} \|v\|_{H^{|\alpha|}(\G)}^2
   =
   \sum_{|\alpha|\le m}
   r^{2(|\alpha|-|\beta|)}
   \sum_{|\beta|\le |\alpha|}
   r^{2|\beta|} \|D^\beta v\|_{L_2(\G)}^2 \\
   &\lesssim
   \sum_{|\alpha|\le m}
   \sum_{|\beta|\le |\alpha|}
   r^{2|\beta|} \|D^\beta v\|_{L_2(\G)}^2
   \lesssim
   \|v\|_{H_r^m(\G)}^2,
\end{align*}
where in the penultimate step we used the fact that $r$ is
bounded above. 
By interpolation we obtain the boundedness of
$E : H_r^s(\G) \to H_r^s(\R^2)$ for all $s>0$,
i.e.,~\eqref{Es}.

To prove \eqref{equiv} we note that $Ev=v$ on $\G$ to obtain
for any $v\in H^s(\G)$ with $s=m+\sigma$
\begin{align*}
   \|v\|_{s,r,\G}^2
   &=
   \sum_{|\alpha|\le m}
   r^{2|\alpha|} \|D^\alpha v\|_{L_2(\G)}^2
   +
   r^{2s} |v|_{s,\G}^2
   =
   \sum_{|\alpha|\le m}
   r^{2|\alpha|} \|D^\alpha(Ev)\|_{L_2(\G)}^2
   +
   r^{2s} |Ev|_{s,\G}^2 \\
   &\le
   \sum_{|\alpha|\le m}
   r^{2|\alpha|} \|D^\alpha(Ev)\|_{L_2(\fR^2)}^2
   +
   r^{2s} |Ev|_{s,\fR^2}^2
   =
   \|Ev\|_{s,r,\fR^2}^2.
\end{align*}
By using~\eqref{scaled s} and~\eqref{Es} we
deduce~\eqref{equiv}.
\end{proof}

In the following we recall and adapt techniques from
\cite{NarcowichWW_05_SBF} to bound the approximation error
appearing in Theorem~\ref{thm_cea}.
We will need the following result; see 
\cite[Theorem 2.12]{NarcowichWW_05_SBF},
\cite[Corollary~4.1]{ArcangeliST_07_ESB}.

\begin{prop} \label{prop2}
Let $k$ be a positive integer and $\sigma\in(0,1]$.
Then there exists $h_0>0$ such that for any $X$ with $h_{X,\G}\le h_0$
and for any $l=0,\ldots,k$
there holds
\[
   |v|_{l,\G}
   \le
   C(k,\sigma)\, h_{X,\G}^{k+\sigma-l} |v|_{k+\sigma,\G}
   \quad
   \forall v\in H^{k+\sigma}(\G)\ \mbox{with}\ v|_X=0.
\]
\end{prop}


In the following, $\lfloor\tau\rfloor$ denotes the largest integer
smaller than or equal to $\tau$.

\begin{lemma} \label{la_IX}
Suppose that assumptions {\rm(A1)} and \eqref{hatPhi} are satisfied.
Then for $s\in[0,{\lfloor\tau\rfloor}]$ there holds
\[
   \|v-I_Xv\|_{H_r^s(\G)}
   \le
   C(s,\tau)
   \left(\frac {h_{X,\G}}{r} \right)^{\tau-s} \|v\|_{H_r^\tau(\G)}
   \quad\forall r\in (0,r_0].
\]
\end{lemma}

\begin{proof}
Let
\[
   E:\; H_r^\tau(\G)\ \to\ H_r^\tau(\R^2)
\]
be a uniformly (in $r$) bounded extension operator,
cf. Lemma~\ref{la_norms}, and let $v\in H_r^\tau(\G)$.

Using that $Ev=v$ on $\G$ and thus $I_XEv=I_Xv=EI_Xv$ on $\G$
(since $X\subset \bar\G$) one finds
that $E v - I_X E v$ is an extension of $v - I_Xv$. Therefore,
the property that $I_X$ is an orthogonal projection in
$\CN_{\Phi_r}$ yields, noting \eqref{scaled},
\begin{align} \label{stab_IX}
   \|v - I_Xv\|_{H_r^\tau(\G)}
   &\lesssim
   \|E_k v - I_X E_k v\|_{H_r^\tau(\fR^2)}
   \simeq
   \|E_k v - I_X E_k v\|_{\CN_{\Phi_r}}
   \nonumber \\
   &\le
   \|E_k v\|_{\CN_{\Phi_r}}
   \simeq
   \|E_k v\|_{H_{r}^\tau(\R^2)}
   \lesssim
   \|v\|_{H_r^\tau(\G)}.
\end{align}
For integer $\ell\le\lfloor\tau\rfloor$,
Proposition~\ref{prop2}, estimate \eqref{equiv}
and stability \eqref{stab_IX} yield
\begin{align*}
   r^\ell |v-I_Xv|_{\ell,\G}
   &\lesssim
   r^\ell h_{X,\G}^{\tau-\ell} \|v-I_Xv\|_{\tau,\G}
   \lesssim
   \left(\frac {h_{X,\G}}{r} \right)^{\tau-\ell} \|v-I_Xv\|_{\tau,r,\G}
   \\
   &\lesssim
   \left(\frac {h_{X,\G}}{r} \right)^{\tau-\ell} \|v-I_Xv\|_{H_r^\tau(\G)}
   \lesssim
   \left(\frac {h_{X,\G}}{r} \right)^{\tau-\ell} \|v\|_{H_r^\tau(\G)}.
\end{align*}
This proves the assertion for integer $s$.
For non-integer $s<\tau$ we interpolate between
$H_r^{\lfloor s\rfloor}(\G)$ and $H_r^{\lfloor
s\rfloor+1}(\G)$,
noting that if $s=\lfloor s\rfloor+\sigma$ with $\sigma\in(0,1)$
be such that $s\le\lfloor\tau\rfloor$ then
$\lfloor s\rfloor+1\le\lfloor\tau\rfloor$.
\end{proof}

\begin{theorem} \label{thm_approx}
Let assumptions {\rm(A1)} and \eqref{hatPhi} be satisfied.
For $t\in(0,{\lfloor\tau\rfloor}]$ with $t+1/2$ being non-integer,
let $r\simeq h_{X,\G}^{1-t/\tau}$.
Then for any $v\in\tilde H^t(\G)$
there exists $\psi\in H_{X,r}$ such that for $0\le s\le t$ there holds
\[
   \|v-\psi\|_{H^s(\G)}
   \le
   C h_{X,\G}^{(t-s)(1-t/\tau)} \|v\|_{\tilde H^t(\G)}.
\]
Here, the constant $C$ is independent of $v$ and $h_{X,\G}$ but may
depend on $s$, $t$ and $\tau$.
\end{theorem}

\begin{proof}
We follow the proof of \cite[Theorem 3.8]{NarcowichWW_05_SBF}.
Let $v\in\tilde H^t(\G)$ be given.
According to \cite[Proposition 3.6]{NarcowichW_04_SDI}
for any $\sigma>0$, there exists a band-limited function
$g_\sigma\in\CB_\sigma:=\{v\in L_2(\R^2);\; \supp(\hat v)\subset B(0,\sigma)\}$
such that (noting that $\|\cdot\|_{H_F^s(\R^2)} \simeq
\|\cdot\|_{H^s(\R^2)}$)
\begin{equation}\label{p0}
   \|v^0-g_\sigma\|_{H^s(\fR^2)} \le C(s,t)
\sigma^{s-t}\|v^0\|_{H^t(\fR^2)}.
\end{equation}
We then define $\psi:=I_Xg_\sigma$ and obtain, together with
Lemma~\ref{la_IX},
\begin{align} \label{p1}
   \|v-I_Xg_\sigma\|_{H^t(\G)}
   &\lesssim
   \|v^0-g_\sigma\|_{H^t(\fR^2)} + \|g_\sigma-I_Xg_\sigma\|_{H^t(\G)}
   \nonumber
   \\
   &\lesssim
   \|v^0-g_\sigma\|_{H^t(\fR^2)} + r^{-t} \|g_\sigma-I_Xg_\sigma\|_{H_r^t(\G)}
   \nonumber
   \\
   &\lesssim
   \|v^0\|_{H^t(\fR^2)}
   +
   r^{-t} \left(\frac hr\right)^{\tau-t}\|g_\sigma\|_{H_r^\tau(\G)}
   \nonumber
   \\
   &\simeq
   \|v\|_{\tilde H^t(\G)}
   +
   r^{-\tau} h^{\tau-t}\|g_\sigma\|_{H_r^\tau(\G)}.
\end{align}
Here, and in the rest of the proof, $h$ denotes $h_{X,\G}$.
Also we used that
\[
   \|v\|_{\tilde H^t(\G)}
   \simeq
   \|v\|_{H^t(\G)}
   \quad\forall v\in \tilde H^t(\G)=H_0^t(\G),
\]
when $t+1/2$ is not an integer, cf. \cite{LionsMagenes}.

Since $g_\sigma$ is a band-limited function there holds
\begin{align*}
   \|g_\sigma\|_{H_r^\tau(\G)}^2
   &\lesssim
   \|g_\sigma\|_{H_r^\tau(\fR^2)}^2
   \simeq
   \|g_\sigma\|_{H_{F,r}^\tau(\R^2)}^2
   \simeq
   \int_{B(0,\sigma)} |\hat g_\sigma(\xi)|^2 (1+r^2|\xi|^2)^\tau\,d\xi
   \\
   &\le
   (1+r^2\sigma^2)^{\tau-t}
   \int_{B(0,\sigma)} |\hat g_\sigma(\xi)|^2 (1+r^2|\xi|^2)^t\,d\xi
   \simeq
   (1+r^2\sigma^2)^{\tau-t}
   \|g_\sigma\|_{H_r^t(\fR^2)}^2.
\end{align*}
Again using \eqref{p0}, we obtain
\[
   \|g_\sigma\|_{H_r^t(\fR^2)}
   \lesssim
   \|g_\sigma\|_{H^t(\fR^2)}
   \le
   \|v^0-g_\sigma\|_{H^t(\fR^2)} + \|v^0\|_{H^t(\fR^2)}
   \lesssim
   \|v^0\|_{H^t(\fR^2)}
   \simeq
   \|v\|_{\tilde H^t(\G)},
\]
so that with the previous estimate,
\[
   \|g_\sigma\|_{H_r^\tau(\G)}
   \lesssim
   (1+r^2\sigma^2)^{(\tau-t)/2}
   \|v\|_{\tilde H^t(\G)}.
\]
Combination with \eqref{p1} yields
\begin{equation}\label{p3}
   \|v-I_Xg_\sigma\|_{H^t(\G)}
   \lesssim
   \left( 1 +
      r^{-\tau} h^{\tau-t}
      (1+r^2\sigma^2)^{(\tau-t)/2}
   \right)
   \|v\|_{\tilde H^t(\G)}.
\end{equation}
By the same arguments and using the same construction we bound
\begin{align} \label{p4}
   \|v-I_Xg_\sigma\|_{L_2(\G)}
   &\le
   \|v^0-g_\sigma\|_{L_2(\fR^2)} + \|g_\sigma-I_Xg_\sigma\|_{L_2(\G)}
   \nonumber
   \\
   &\lesssim
   \sigma^{-t} \|v\|_{\tilde H^t(\G)}
   +
   \left(\frac hr\right)^\tau \|g_\sigma\|_{H_r^\tau(\G)}
   \nonumber
   \\
   &\lesssim
   \sigma^{-t} \|v\|_{\tilde H^t(\G)}
   +
   r^{-\tau} h^\tau (1 + r^2\sigma^2)^{(\tau-t)/2} \|v\|_{\tilde H^t(\G)}
   \nonumber
   \\
   &=
   (\sigma^{-t} + r^{-\tau} h^\tau (1 + r^2\sigma^2)^{(\tau-t)/2})
   \|v\|_{\tilde H^t(\G)}.
\end{align}
Now choosing $\sigma=1/r$ in \eqref{p3} and \eqref{p4}, and
taking $r=h^\alpha$ with $\alpha=1-t/\tau$, we obtain
\[
   \|v-I_Xg_\sigma\|_{H^t(\G)}
   \lesssim
   \left( 1 + r^{-\tau} h^{\tau-t} \right)
   \|v\|_{\tilde H^t(\G)}
   =
   2\, \|v\|_{\tilde H^t(\G)}
\]
and
\[
   \|v-I_Xg_\sigma\|_{L_2(\G)}
   \lesssim
   (r^t + r^{-\tau}h^\tau) \|v\|_{\tilde H^t(\G)}
   =
   (h^{\alpha t} + h^t) \|v\|_{\tilde H^t(\G)}
   \lesssim
   h^{\alpha t} \|v\|_{\tilde H^t(\G)}.
\]
Interpolation gives
\begin{align*}
   \|v-I_Xg_\sigma\|_{H^s(\G)}
   &\lesssim
   \|v-I_Xg_\sigma\|_{H^t(\G)}^{s/t}
   \|v-I_Xg_\sigma\|_{L_2(\G)}^{1-s/t}
   \\
   &\lesssim
   h^{\alpha t(1-s/t)} \|v\|_{\tilde H^t(\G)}
   =
   h^{(1-t/\tau)(t-s)} \|v\|_{\tilde H^t(\G)}.
\end{align*}
This proves the theorem.
%
\end{proof}

\mysection{Conclusions}{sec_con}

Before presenting numerical experiments let us draw some conclusions.

Based on Assumptions (A1)--(A4) we have proved the quasi-optimal convergence
of the discrete scheme \eqref{BEM} (cf.~Theorem~\ref{thm_cea1}). The best approximation
of the Lagrangian multiplier is taken in the natural space of order $-1/2$ on
the domain of definition of the Lagrangian multiplier (outside $\G$).
The best approximation of the sought solution $\phi$ is measured in the natural
space of order $1/2$, on an extended domain and taken among scaled radial basis
functions of the constrained space, i.e. with piecewise mean value zero on the
extension. With Theorem~\ref{thm_cea} we managed to replace the latter term
(best approximation of $\phi$) with the best approximation in the unconstrained
space on the original domain $\G$. The price to pay is an additional stability
term that measures the approximant on the extension, where the unknown solution
vanishes. We were able to bound the best approximation error on the original
domain by a term that shows a convergence order that is close to the one of a
standard boundary element method when the regularity $\tau$ of the native space
becomes large, cf. Theorem~\ref{thm_approx}.

We were unable to show an appropriate bound for the stability term
$|\psi|_{H^t(\Gck)}$ in Theorem~\rref{thm_cea}. The natural tool to estimate
$|\psi|_{H^t(\Gck)}=|I_Xg_\sigma|_{H^t(\Gck)}$ is switching to the norm
of the native space. (Recall that $g_\sigma$ is a band-limited function that
approximates the solution of the integral equation.)
This switch produces a factor of $r^{-\tau}$ for the $L_2(\G)$-norm
of $g_\sigma$ which we cannot control efficiently as we can with the other
higher order term $|g_\sigma|_{H^\tau_r(\G)}$.

Our numerical results for the choice $k\simeq r\simeq h_{X,\G}^{1-t/\tau}$
and $t=1$ (cf. Figures~\ref{fig_error0}--\ref{fig_error2})
indicate that $|\psi|_{H^t(\Gck)}$
asymptotically behaves, for the discrete solution $\phi_N$, exactly as the
best approximation error. Note that for the chosen parameters,
\be \label{stab_term}
   \delta(\epsilon,t,k,r)
   |\psi|_{H^t(\Gck)}
   \simeq
   k^{t-1/2-\epsilon} |\psi|_{H^t(\Gck)}
   \simeq
   h_{X,\G}^{(t-1/2-\epsilon)(1-t/\tau)} |\psi|_{H^t(\Gck)},
\ee
so that it is enough to have boundedness of $|\psi|_{H^t(\Gck)}$
to obtain the optimal approximation order, cf. the final estimate
\eqref{main_con} below.
Based on the assumption that the minimizer $\psi\in H_{X,r}$ of
$\|\phi-\psi\|_{H^{1/2+\epsilon}(\G)}$ has bounded semi-norm
$|\psi|_{H^t(\Gck)}$, i.e.,
\be \label{stab}
   \inf_{\psi\in H_{X,r}}
   \Bigl(
      \|\phi-\psi\|_{H^{1/2+\epsilon}(\G)}
      +
      \delta(\epsilon,t,k,r)
      |\psi|_{H^t(\Gck)}
   \Bigr)
   \simeq
   \inf_{\psi\in H_{X,r}}
   \|\phi-\psi\|_{H^{1/2+\epsilon}(\G)},
\ee
let us deduce a final error estimate.
Under assumption \eqref{stab}, Theorem~\rref{thm_cea} gives
\be \label{t0}
   \|\phi^0-\phi_N\|_{\tHp(\tGk)}
   \le
   C(\epsilon)
   \inf_{\psi\in H_{X,r}}
   \|\phi-\psi\|_{H^{1/2+\epsilon}(\G)}
   +
   \inf_{\mu\in M_k} \|\lam-\mu\|_{\tHm_\g(\Gck)}.
\ee
By Theorem~\rref{thm_approx} we bound
\be \label{t1}
   \inf_{\psi\in H_{X,r}}
   \|\phi-\psi\|_{H^{1/2+\epsilon}(\G)}
   \lesssim
   h_{X,\G}^{(t-1/2-\epsilon)(1-t/\tau)}
   \|\phi\|_{\tilde H^t(\G)}.
\ee
There holds $\lam=(W\phi^0)|_{\Gc}$, cf. Proposition~\rref{prop1}.
Now, $\phi^0\in\tilde H^t(\tG)$ so that $\lam|_{\Gck}\in H^{t-1}(\Gck)$ by
the mapping properties of $W$, cf.~\rcite{McLean_00_SES}.
A standard approximation result
(see {\rm\cite[Lemma~2.3]{BespalovH_11_NHC}} for the $p$-result on an element;
this immediately generalizes to the present case) yields
\be \label{t2}
   \inf_{\mu\in M_k} \|\lam-\mu\|_{\tHm_\g(\Gck)}
   \lesssim
   k^{t-1/2} \|\lam\|_{H^{t-1}(\Gck)}.
\ee
Combining \eqref{t1} and \eqref{t2} with
$k\simeq h_{X,\tGk}^{1-t/\tau}$, and using the estimate \eqref{t0}, this leads to
\begin{equation} \label{main_con}
   \|\phi-\phi_N\|_{H^{1/2}(\G)}
   \le
   \|\phi^0-\phi_N\|_{\tHp(\tGk)}
   \le
   C\, h_{X,\G}^{(t-1/2-\epsilon)(1-t/\tau)}
   \quad \epsilon>0, t\in(1/2,1].
\end{equation}
The constant $C$ would depend on $t$, $\epsilon$ and $\tau$
but not on $X$ and $h_{X,\G}$ under the assumptions made.

\mysection{Numerical results}{sec_num}

We consider the model problem \eqref{IE} with $\Gamma=(0,1)\times (0,1)$
and $f=1$.
The nodes of $X$ are distributed uniformly on $\bar\G$ where $\G$ is extended
by a strip of uniform squares (with side length $k$) as in Figure~\ref{fig_mesh}.
The support of the scaled radial basis functions is indicated in one case.
The setup is selected so that assumptions (A1)--(A4) are satisfied and with mesh norm
$h=h_{X,\G}^{1-1/\tau}\simeq r\simeq k$, according to \eqref{main}.
We use scaled radial basis functions with the radial basis functions defined in
\cite{Wendland_98_EEI} (for the case $d=2$ there which corresponds to functions in $\R^2$).
These functions satisfy relation \eqref{hatPhi} with $\tau=3/2+m$ where $m$ is the parameter
of the corresponding regularity $C^{2m}$. The degree of the corresponding univariate
function is $2+3m$.

\begin{figure}[htb]
\begin{center}
\includegraphics[width=0.4\textwidth]{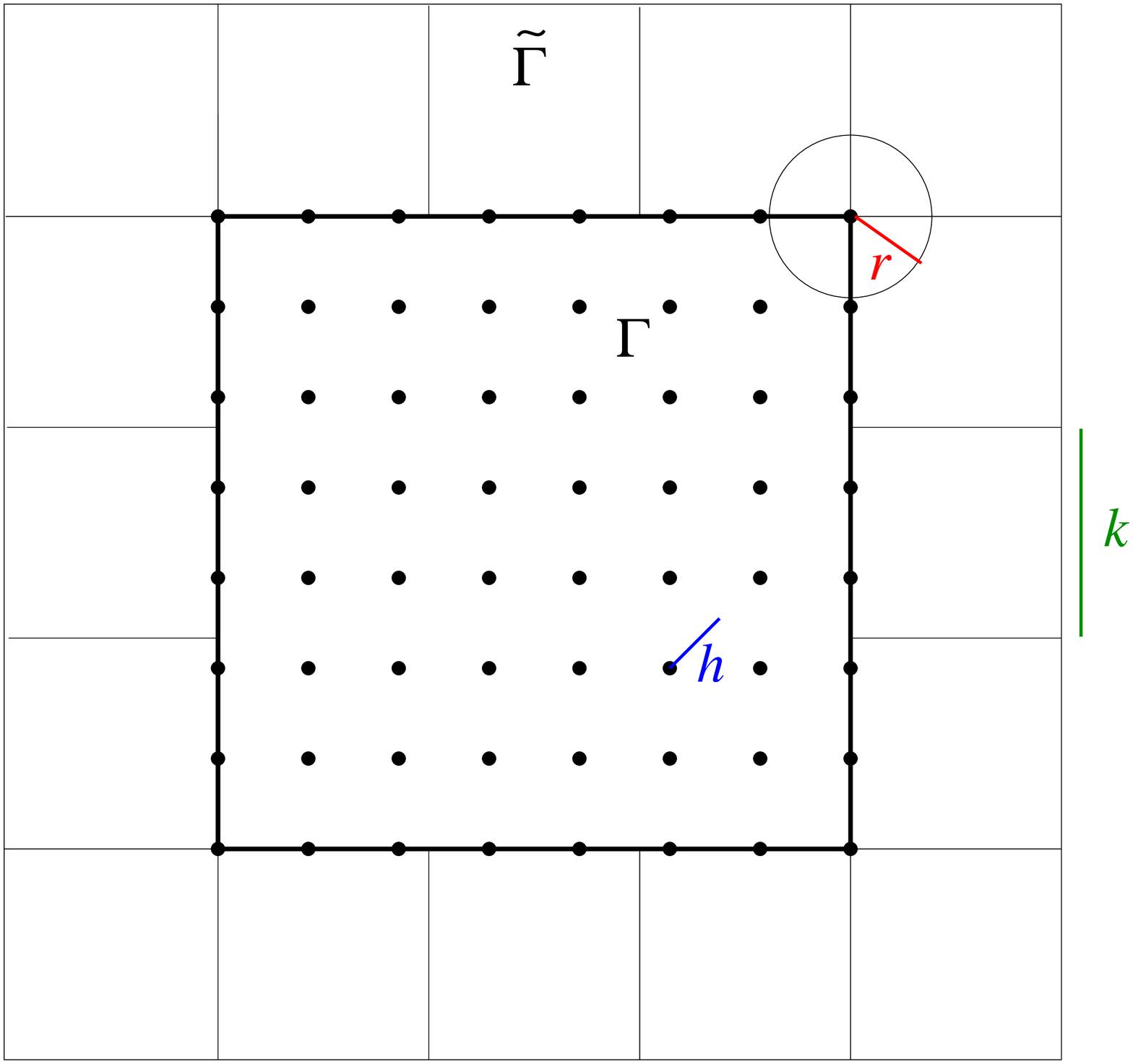}
\end{center}
\caption{Uniformly distributed nodes $X$ on $\bar\G$, and uniform mesh generating
         $\Gck$ and extending $\G$ to $\tilde\G$.}
\label{fig_mesh}
\end{figure}

We calculate the errors in an approximating (and heuristic) manner.
For a conforming method, by making use of the symmetry of the hypersingular operator,
one obtains
\begin{align*}
   \|\phi^0-\phi_N\|_{\tHp(\tG)}^2
   &\simeq
   \<W(\phi^0-\phi_N),(\phi^0-\phi_N)\>_{\tG}
   \\
   &=
   \<W\phi^0,\phi^0\>_{\tG} - \<W\phi_N,\phi_N\>_{\tG}
   =
   \<W\phi,\phi\>_{\G} - \<W\phi_N,\phi_N\>_{\tG}.
\end{align*}
The last term is available through the stiffness matrix of the problem and the
term $\<W\phi,\phi\>_{\G}$ can be approximated by extrapolation,
cf.~\cite{ErvinHS_93_hpB}. In our case of a non-conforming (or mixed) approximation,
this calculation has a perturbation which is due to the term $\lam=(W\phi^0)|_{\Gck}\not=0$.
Since we do not know the exact solution $\phi$ so that $\lam$ cannot be calculated
we approximate the relative error in the energy norm by the expression
\begin{equation} \label{app_err}
   \Bigl(\|\phi\|_{\rm ex}^2 - \<W\phi_N,\phi_N\>_{\tG}\Bigr)^{1/2}/\|\phi\|_{\rm ex}.
\end{equation}
Here, $\|\phi\|_{\rm ex}^2$ denotes the extrapolated value substituting $\<W\phi,\phi\>_{\G}$.

For different values of $\tau$,
we present the approximated errors on double logarithmic scales along with the expected
convergence rates (upper limit) according to \eqref{main}.
For comparability we use the same scales in all the figures.
Table~\ref{tab1} lists the values of $\tau$ with corresponding figure number, data
(regularity $C^{2m}$ and polynomial degrees as mentioned before) and the limit
$1/2(1-1/\tau)$ for the expected convergence rates.

\begin{table}[htb]
\begin{center}
\begin{tabular}{cc|ccc}
   $\tau$ & figure no. & regularity & polynomial degree &
            exp. conv. rate $\frac 12(1-\frac 1\tau)$\\ \hline
   $3/2$  &  \ref{fig_error0}   & $C^0$  &  $2$  &  $1/6$\\
   $5/2$  &  \ref{fig_error1}   & $C^2$  &  $5$  &  $3/10$\\
   $7/2$  &  \ref{fig_error2}   & $C^4$  &  $8$  &  $5/14$
\end{tabular}
\caption{Values of $\tau$ in the numerical experiments with corresponding data,
         expected convergence rates and figures.}
\label{tab1}
\end{center}
\end{table}

Figures~\ref{fig_error0}--\ref{fig_error2} confirm quite precisely
the predicted convergence. The lines indicated by ``error'' give the approximate
relative errors calculated by \eqref{app_err} whereas the lines ``stab term''
give the values of the stability term $\delta(\epsilon,t,k,r)|\psi|_{H^t(\Gck}$
for $\epsilon=0$, $t=1$ (the limit of the regularity) and $\psi$ the
calculated RBF approximation, cf. \eqref{stab_term}. The lines ``expected''
plot the expected convergence rates as listed in Table~\ref{tab1}. They correspond
to our conclusion \eqref{main_con} for $\epsilon=0$ and $t=1$.
All results indicate that the errors have the predicted convergence
rates and that the stability term \eqref{stab_term}
fulfills our conjecture \eqref{stab}.

We have implemented the method
by numerical integration with an overkill of number of integration nodes.
We used transformation to polar coordinates so that the
singularity from the fundamental solution cancels in the diagonal entries of the stiffness matrix.
Nevertheless, note that the polynomial degrees for larger values of $\tau$ are large
($8$ in the case $\tau=\frac 72$) which makes their implementation a non-trivial task.

In the case $\tau=\frac 72$ there is a large pre-asymptotic range
(Figure~\ref{fig_error2}). Note also that for $\tau=\frac 32$ we were not able to
calculate the stability term for the whole range of unknowns (about 30,000).
In this case the radial basis functions are only continuous and the numerical
calculation of the $H^1$-semi-norm becomes unstable.
%

\begin{figure}[htb]
\begin{center}
\includegraphics[width=0.7\textwidth]{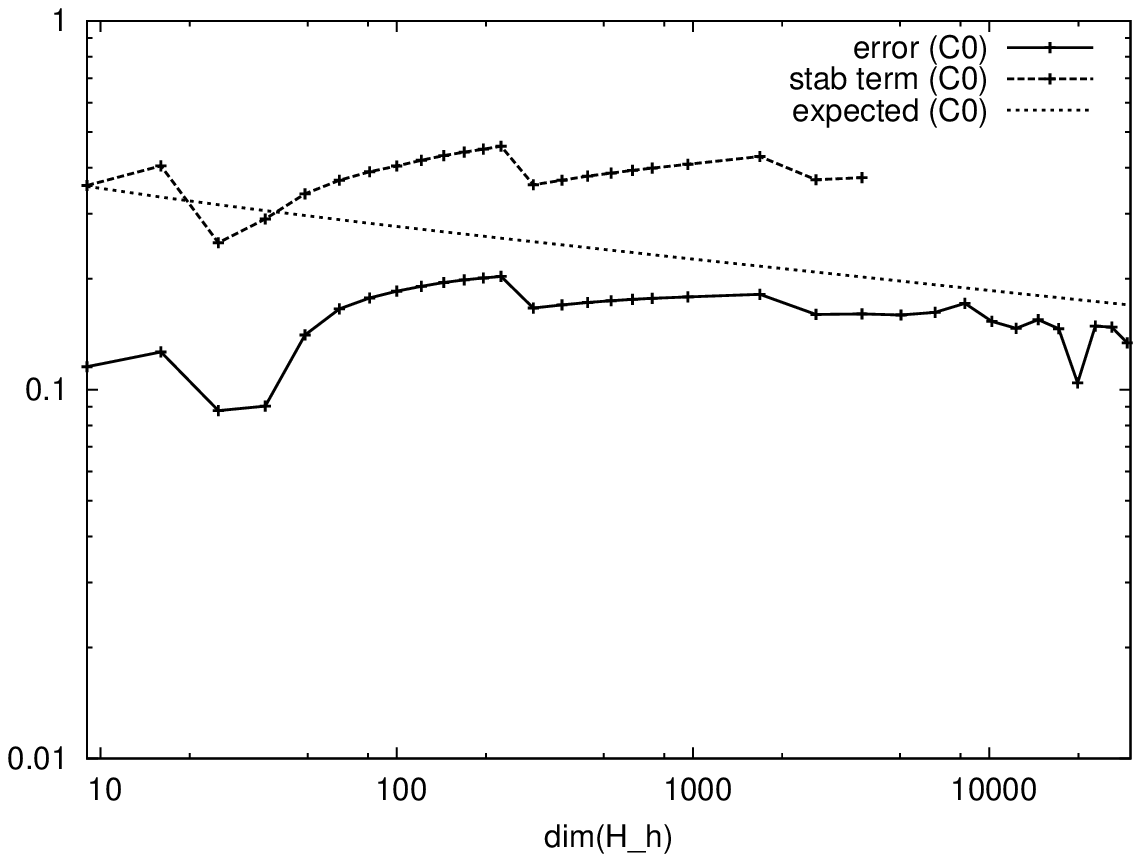}
\end{center}
\caption{Relative error and theoretical convergence rate for $\tau=1.5$.}
\label{fig_error0}
\end{figure}

\begin{figure}[htb]
\begin{center}
\includegraphics[width=0.7\textwidth]{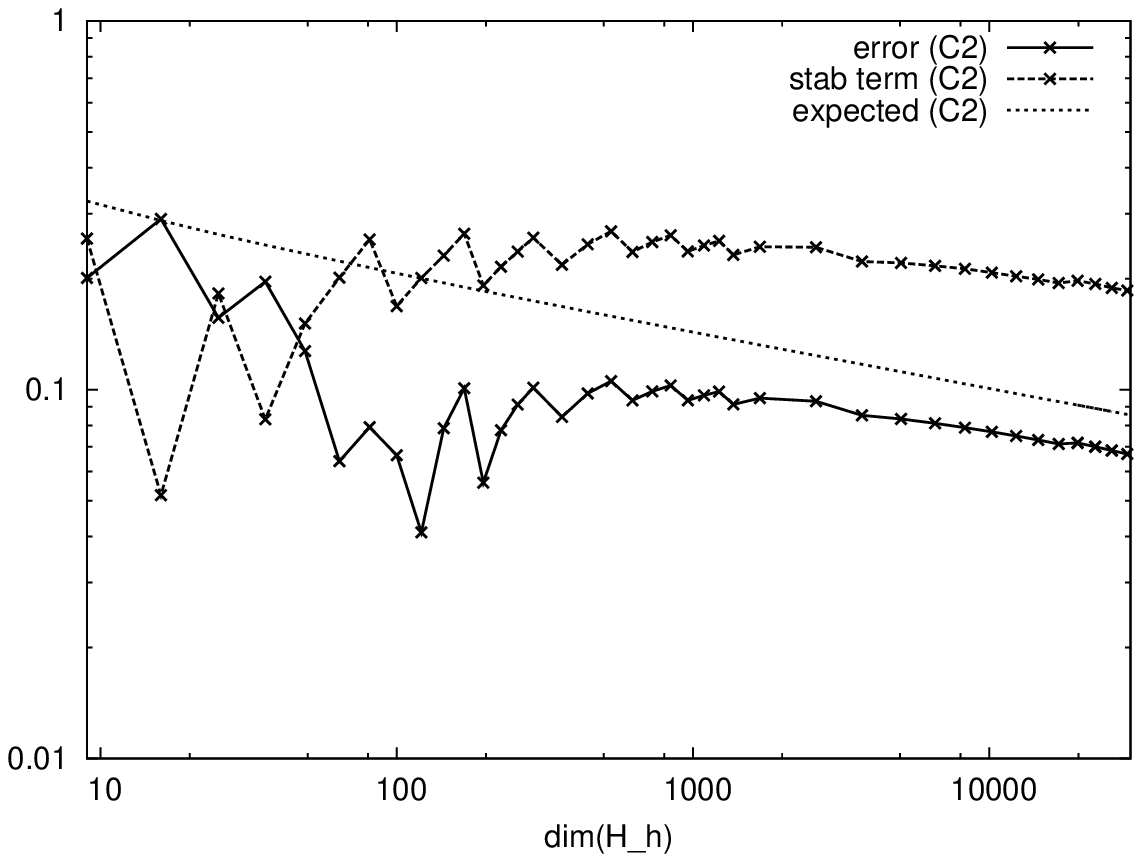}
\end{center}
\caption{Relative error and theoretical convergence rate for $\tau=2.5$.}
\label{fig_error1}
\end{figure}

\begin{figure}[htb]
\begin{center}
\includegraphics[width=0.7\textwidth]{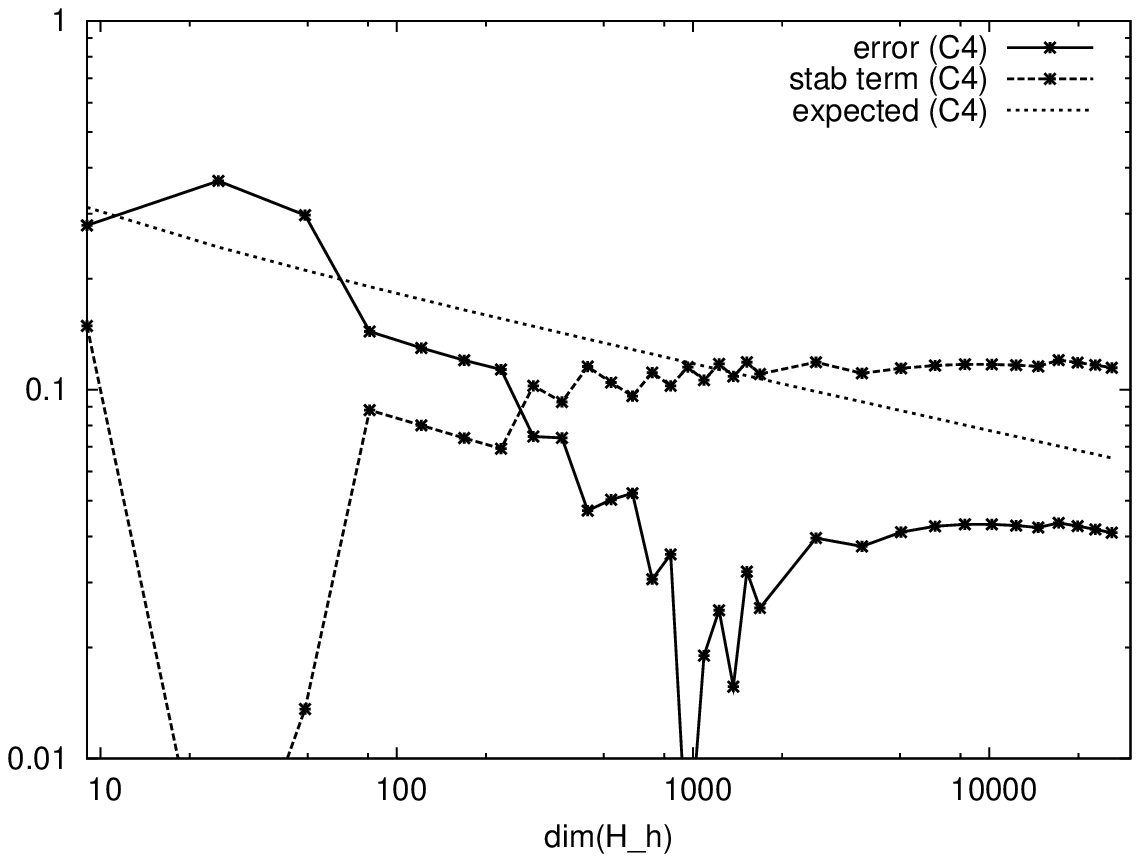}
\end{center}
\caption{Relative error and theoretical convergence rate for $\tau=3.5$.}
\label{fig_error2}
\end{figure}

%

\bigskip
\clearpage
\noindent{\bf Acknowledgment.}
A significant part of this work has been done while N.H. was visiting
the School of Mathematics at The University of New South Wales in Sydney.
Their hospitality is gratefully acknowledged.



\end{document}